\documentclass[a4paper,12pt]{amsart}

\usepackage{etoolbox}

\makeatletter
\patchcmd{\@setaddresses}{\indent}{\noindent}{}{}
\patchcmd{\@setaddresses}{\indent}{\noindent}{}{}
\patchcmd{\@setaddresses}{\indent}{\noindent}{}{}
\patchcmd{\@setaddresses}{\indent}{\noindent}{}{}
\makeatother

\usepackage{float}

\usepackage{lmodern}
\usepackage[T1]{fontenc}
\usepackage{microtype}

\usepackage[english]{babel}
\usepackage{amssymb,amscd,url,stmaryrd,mathtools}
\usepackage{paralist,stackrel}

\usepackage[mathscr]{eucal}

\usepackage{ latexsym, amsfonts, amsmath, amscd,xy}
\usepackage{graphicx,textcomp}
\usepackage{xcolor}

\usepackage{tikz}\usetikzlibrary{matrix}
\usepackage{tikz-cd}
\usepackage{hyperref}

\usepackage{url}

\numberwithin{equation}{section}
\mathtoolsset{mathic=true}
\setdefaultenum{\upshape (i)}{}{}{}

\usepackage{xcolor}

\DeclareRobustCommand{\qee}{
  \ifmmode \mathqee
  \else
    \leavevmode\unskip\penalty9999 \hbox{}\nobreak\hfill
    \quad\hbox{\qeesymbol}
  \fi
}
\newcommand{\mathqee}{\quad\hbox{\qeesymbol}}
\newcommand{\qeesymbol}{\ensuremath\diamondsuit}

\newcommand{\Lie}[1]{\operatorname{\textsl{#1}}}
\newcommand{\lie}[1]{\operatorname{\mathfrak{#1}}}
\newcommand{\GL}{\Lie{GL}}
\newcommand{\gl}{\lie{gl}}
\newcommand{\sln}{\lie{sl}}

\newcommand{\so}{\lie{so}}

\newcommand{\sP}{\lie{sp}}
\newcommand{\su}{\lie{su}}
\newcommand{\un}{\lie u}
\newcommand{\g}{\lie g}

\newcommand{\bmf}{\lie b}

\newcommand{\kf}{\lie k}

\newcommand{\tf}{\lie t}

\newcommand{\SL}{\Lie{SL}}
\newcommand{\SU}{\Lie{SU}}
\newcommand{\Un}{\Lie{U}}

\newcommand{\C}{{\mathbb C}}
\newcommand{\HH}{{\mathbb H}}
\newcommand{\PP}{{\mathbb P}}
\newcommand{\R}{{\mathbb R}}

\newcommand{\impl}{\textup{impl}}

\newcommand{\symp}{{\sslash}} 
\newcommand{\hkq}{{\sslash\mkern-6mu/}}
\newcommand{\Bigsymp}{\Big/\mkern-11mu\Big/}

\DeclareMathOperator{\Hom}{Hom}
\DeclareMathOperator{\Id}{Id}

\DeclareMathOperator{\rank}{rank }

\DeclareMathOperator{\Spec}{Spec}

\DeclareMathOperator{\syco}{sc}
\DeclareMathOperator{\hoco}{hc}
\DeclareMathOperator{\cosyco}{csc}
\DeclareMathOperator{\As}{As}
\DeclareMathOperator{\gr}{gr}

\DeclarePairedDelimiter{\norm}{\lVert}{\rVert}

\usepackage{tikz} 

\usetikzlibrary{decorations.pathreplacing} 
\usetikzlibrary{decorations.pathmorphing} 
\usetikzlibrary{decorations.markings} 
\usetikzlibrary{arrows} 
\usetikzlibrary{shapes} 
\usetikzlibrary{matrix} 
\usetikzlibrary{positioning}

\tikzset{cross/.style={cross out, draw=black}} 
\tikzset{circ/.style={circle,fill=white,draw=black}} 
\tikzset{hasse/.style={circle, fill,inner sep=2pt}} 
\tikzset{h/.style={circle, fill,inner sep=2pt}} 
\tikzset{ns/.style={circle, draw,inner sep=2pt}} 
\tikzset{dot/.style={circle,draw,fill=black}} 
\tikzset{gauge/.style={circle,draw}} 
\tikzset{flavor/.style={regular polygon,regular polygon sides=4, draw}} 
\tikzset{doublearrow/.style={ draw=black!75, color=black!75, thick, double distance=3pt, }} 
\tikzset{thirdline/.style={ draw=black!75, color=black!75, thick, }} 
\tikzset{middlearrow/.style={ decoration={markings, mark= at position 0.5 with {\arrow{#1}} , }, postaction={decorate} } } 
\tikzset{subquiver/.style={circle,draw,dashed,inner sep=0pt,minimum size=1cm}}

\newcommand{\aff}{\operatorname{Aff}}

\DeclareFontFamily{U}{mathb}{\hyphenchar\font45}
\DeclareFontShape{U}{mathb}{m}{n}{
      <5> <6> <7> <8> <9> <10> gen * mathb
      <10.95> mathb10 <12> <14.4> <17.28> <20.74> <24.88> mathb12
      }{}
\DeclareSymbolFont{mathb}{U}{mathb}{m}{n}
\DeclareFontSubstitution{U}{mathb}{m}{n}
\DeclareMathSymbol{\righttoleftarrow}      {3}{mathb}{"FD}

\newcommand{\actson}{\reflectbox{$\righttoleftarrow$}}

\newtheorem{lemma}{Lemma}[section]
\newtheorem{proposition}{Proposition}[section]
\newtheorem{definition}{Definition}[section]

\newtheorem{remark}{Remark}[section]

\begin{document}

\title{Implosion, Contraction and Moore-Tachikawa}

\author{Andrew Dancer}
\address{Andrew Dancer\\ Jesus College\\
Oxford
OX1 3DW,
United Kingdom}
\email{dancer@maths.ox.ac.uk }

\author{Frances Kirwan}
\address{Frances Kirwan\\ New College\\
Oxford
OX1 3BN,
United Kingdom}\email{kirwan@maths.ox.ac.uk} 

\author{Johan Martens}

\address{Johan Martens\\ School of Mathematics and Maxwell Institute,\linebreak University of Edinburgh,\\
James Clerk Maxwell Building, Edinburgh EH9 3FD, United Kingdom}
\email{johan.martens@ed.ac.uk}
 

\maketitle

\begin{abstract}
We give a survey of the implosion construction, extending some of its aspects
relating to hypertoric geometry from type $A$ to a general reductive group, and interpret it in the context of the Moore-Tachikawa
category.
We use these ideas to discuss how the contraction construction in symplectic geometry can be
generalised to the hyperk\"ahler or complex symplectic situation.
\end{abstract}

\vspace{.5cm}
\dedicatory{\noindent\emph{Dedicated to Oscar Garcia-Prada on the occasion of his 60th birthday.}}

\section{Introduction} \label{intro}
The symplectic implosion construction of Guillemin, Jeffrey and Sjamaar \cite{GJS} is an Abelianisation
construction for Hamiltonian group actions. Given a symplectic space $M$ with such an action of a compact Lie group $K$, the implosion
$M_{\impl}$
carries an action of the maximal torus
$T$ of $K$, such that the symplectic reduction of the implosion $M_{\impl}$
  by $T$ at a level in the closed positive Weyl chamber is the same as the reduction of $M$ by $K$.

\medskip
If $M$ actually arises as a K\"ahler variety, projective over an affine, equipped with a linearised action of $K_{\C}$ and compatible K\"ahler structure, 
$M_{\impl}$ can also be understood as the non-reductive geometric invariant theory quotient of $M$ by $U$, a maximal unipotent subgroup of $G=K_{\C}$ \cite{kirwan.2011}.

\medskip
Implosion was generalised to the hyperk\"ahler situation in \cite{DKS} in the
$\SU(n)$ case. In the general case, starting from an arbitrary compact connected Lie group, work of Ginzburg and Riche \cite{GinzburgRiche} shows we have at least a complex symplectic version of the implosion (though the hyperk\"ahler structure has not been shown yet, it is expected to exist, and we will refer to it as such).
The hyperk\"ahler version takes inspiration from both the (real) symplectic and algebro-geometric version, but is overall a step up in complexity.  This is parallel to other operations in equivariant geometry that come in three flavours, such as symplectic reduction \cite{marsden-weinstein.1974}, geometric invariant theory \cite{GIT}, and hyperk\"ahler reduction \cite{HKLR.1987}, or (for abelian actions) symplectic and algebraic cutting \cite{lerman.1995,edidin-graham.1998} and hyperk\"ahler modifications \cite{dancer-swann.2006}.  

\medskip
In this note we expand this viewpoint by introducing the complex symplectic version of contraction, which so far was understood only in an algebraic \cite{popov.1986} or real symplectic \cite{HMM} context
(see \S 5 for a survey of these constructions). Unlike reduction or implosion, contraction does not change the dimension of the space one starts from, but gives rise to an additional action of the maximal torus, commuting with the actions of $K$ or $G$.

\medskip
It is common to all these operations that they can be reduced to a universal version, which is the result of applying them to the cotangent bundle $T^*K$ (in the real symplectic case), the complex group $G$ (in the algebraic case), or the complex cotangent bundle $T^*G$ (in the complex symplectic case).  The resulting spaces are often familiar -- the symplectic reductions of $T^*K$ are coadjoint $K$ orbits (or flag varieties for $G$), for example.
For abelian $K$ the symplectic cuts of $T^*K$, or algebraic cuts of $G$, give toric varieties, and the hyperk\"ahler modifications of $T^*G$ are the hypertoric varieties.

\medskip
This in turn means that they naturally fit in with the concept of Moore-Tachikawa categories, whose objects are groups, morphisms are spaces with group actions, and composition is done by reduction of products.  Indeed, the universal implosions are morphisms in this category between a group and its maximal torus, and the universal contraction is the composition of this morphism with itself, suitably interpreted.

\section{Hyperk\"ahler implosion}
In this section we survey symplectic and hyperk\"ahler implosion.

Let $T$ be a maximal torus for a compact Lie group $K$, with Lie algebra $\tf \subset \kf$.  By using the weight decomposition for the adjoint representation, $\tf$ has a canonical complement, and hence we have ${\tf^*\subset \kf^*}$, and also $\tf^*_+\subset \kf^*$ for the choice of a positive Weyl chamber $\tf^*_+$.  If $M$ is a (real) Hamiltonian $K$-manifold, with moment map $\mu:M\rightarrow \kf^*$, the symplectic implosion $M_{\impl}$ is (as a set) $\mu^{-1}(\tf^*_+)/\thicksim$, where $$x\sim y \Leftrightarrow \mu(x)=\mu(y) \ \text{and}\ y \in [K_{\mu(x)}, K_{\mu(x)}].x.$$
It is typically not a smooth manifold, but a stratified Hamiltonian space -- the residual map to $\tf^*_+\subset\tf^*$ is the moment map for a $T$-action.  Applied to $T^*K$ (as a $K$-space with the action induced from $K$ acting on itself as $k.h=hk^{-1}$), this operation gives the universal implosion \begin{equation}\label{implequiv}(T^*K)_{\impl}=K\times\tf^*_+/\thicksim.\end{equation} Besides the $T$-action, $(T^*K)_{\impl}$ still has the residual $K$-action that comes from multiplication on the left, since if a space has two commuting actions by compact groups $K$ and $L$, the implosion by $K$ still has an action of $L$, and moreover implosion and symplectic reduction commute: \begin{equation}\label{commute}M_{K-\impl}\symp L\cong (M\symp L)_{K-\impl}.\end{equation}

\medskip
This last property also explains the relevance of the universal implo\-sion: for any symplectic $K$-manifold, one has a canonical equivariant iso\-morphism 
\begin{equation}\label{trivial}M\cong \left(M\times T^*K\right)\symp \Delta K,\end{equation} where the quotient in the right-hand side uses the diagonal action of $K$ acting on $M$ and on $T^*K$ by multiplication from the right -- this quotient space still inherits a $K$-action coming from the left action of $K$ on $T^*K$.

\medskip
Also for GIT quotients and complex symplectic reduction, one has an analogue of (\ref{trivial}), with the role of $T^*K$ now played by respectively $G=K_\C$ or $T^*G$ (for simplicity of notation we shall indicate both symplectic reduction and GIT quotients by $\symp$, and hyperk\"ahler and complex symplectic reduction by $\hkq$).  Remark that this is not quite true for the metric aspects in K\"ahler or hyperk\"ahler quotients: there is no known metric on $G$ or $T^*G$ that would make the analogues of (\ref{trivial}) an isometry.

\medskip 
Because of the commutativity of implosion and reduction, as in (\ref{commute}), when applying the former to both sides of (\ref{trivial}), we obtain
$$M_{\impl}\cong (M\times (T^*K)_{\impl})\symp \Delta K,$$ 
which shows how to recover any implosion from the universal implosion.  

 \medskip
 The universal symplectic implosion $(T^*K)_{\impl}$ can be identified (as a stratified $K\times T$ space) with $G\symp U =\overline{G/U}^{\aff}=\Spec\left(k[G]^U\right)$, for $U$ the unipotent radical of the Borel subgroup of $G$ determined by the maximal torus $T_{\C}$ and the positive Weyl chamber $\tf_+^*$.  Even though $U$ is not a reductive group, the invariant ring $k[G]^U$ is still finitely generated, which makes this definition sensible.

 \medskip
 This is done as follows (see Section 6 in \cite{GJS} and Appendix in \cite{HMM}): given the choice of a finite set of generators $\Pi$ for the semigroup of dominant weights of $G$, one can embed \begin{equation}\label{G/Uembed} G\symp U \hookrightarrow E=\oplus_{\varpi\in \Pi} V_{\varpi}.\end{equation} Here $V_{\varpi}$ is the irreducible representation of $G$ with highest weight $\varpi$, and we moreover choose highest-weight vectors $v_{\varpi}\in V_{\varpi}$ and $K$-invariant Hermitian inner products on the $V_{\varpi}$ such that $\norm{v_{\varpi}}=1$.  The image of $G\symp U$ is now the closure of the $G$-orbit of $\sum_{\varpi\in\Pi} v_{\varpi}$, and $G\symp U$ inherits a K\"ahler structure from the embedding in $E$.
 
 \medskip
 We can also consider an extra action of $T_{\C}$ on $E$ (commuting with the $G$ action), acting diagonally on $V_{\varpi}$ with weight $-\varpi$, and look at the affine toric variety inside $G\symp U$ that is the closure of the $T_{\C}$-orbit of $\sum_{\varpi\in\Pi} v_{\varpi}$.  The image of the $T$-moment map of this toric variety is $-\tf^*_+$, and the non-negative part of the toric variety provides a section of this moment map.  We can now use minus this section to create a subjective map from $K\times \tf_+^*$ into $G\symp U\subset E$, which (using (\ref{implequiv})) induces an isomorphism of symplectic stratified spaces $(T^*K)_{\impl}$.

 \medskip
 This identification is the basis for the correspondence between sym\-plectic reduction and quotients in non-reductive GIT, though this only concerns those quotients on the algebraic side where the action of the group $U$ extends to one of $G$, in which case the invariant ring will be finitely generated.

 \medskip
 When it comes to a hyperk\"ahler or complex symplectic analogue of implosion, as it too commutes with complex symplectic reduction, the key case is again the implosion of $T^*G$. We recall that this
 space is complex symplectic with $G \times G$-action, and in fact, by the work of Kronheimer \cite{Kronheimer:cotangent}, is even known to be hyperk\"ahler. One way to view its implosion is as the complex symplectic reduction of $T^*G$ by $U$, acting from the right.  If we identify $T^*G\cong G\times \g^*$, then the complex moment map for the $U$-action is given by the projection of the second factor on $\un^*$, where $\un$ is the Lie algebra of $U$.  The level set is the product affine variety $G\times \un^{\circ}$ (where $\un^{\circ}\subset \g^*$ is the annihilator of $\un$), which has a $G\times U$-action, but not a $G\times G$-action, and as a result it is not directly clear that the invariant ring $k[G\times \un^{\circ}]^U$ is finitely generated.  Alternatively, remark that if $(G\times \un^{\circ})\symp U=\Spec\left(k[G\times \un^{\circ}]^U\right)$ were well-defined as an affine variety, it would contain the cotangent space $T^*(G/U)$, and hence the question is whether $T^*(G/U)$ is a quasi-affine variety.

 \medskip
 For $K=\SU(n)$, the question of finite generation was answered affirmatively by Dancer, Kirwan and Swann in \cite{DKS}, where it was shown that the resulting variety could also be constructed through quivers as a hyperk\"ahler quotient of a vector space by a compact group (which in particular showed that the space is not just complex symplectic, but in fact hyperk\"ahler on its smooth locus).  In particular, the following quiver diagram is considered: 
 \begin{center}
 \begin{tikzcd}
 {\color{lightgray}0 } \ar[dd,color=lightgray,xshift=.1cm]& {\color{lightgray}0 }\ar[dd,color=lightgray,xshift=.1cm]&\dots &{\color{lightgray}0 } \ar[dd,color=lightgray,xshift=.1cm] & \C^n \ar[dd,xshift=.1cm,"\beta_{n-1}"]\\ & & & & \\
 \C \ar[r,yshift=.1cm,"\alpha_1"]\ar[uu,color=lightgray,xshift=-.1cm] & \C^2 \ar[r,yshift=.1cm,"\alpha_2"]\ar[l,yshift=-.1cm,"\beta_1"]\ar[uu,color=lightgray,xshift=-.1cm] & \dots \ar[r, yshift=.1cm, "\alpha_{n-3}"]\ar[l,yshift=-.1cm,"\beta_2"] & \C^{n-2} \ar[r,yshift=.1cm,"\alpha_{n-2}"] \ar[l,yshift=-.1cm,"\beta_{n-3}"]\ar[uu,color=lightgray,xshift=-.1cm]&\C^{n-1}. \ar[l,yshift=-.1cm,"\beta_{n-2}"] \ar[uu,xshift=-.1cm,"\alpha_{n-1}"]\\
 \end{tikzcd}
 \end{center}
 (We follow the presentation of the quiver given in \cite{wang.2021} here, including the vacuous light-coloured part, to indicate this can be understood as a special case of a framed and doubled quiver, as in the context of Nakajima quiver varieties.)
 One can take the affine hyperk\"ahler quotient (at zero for all moment maps) of $$M=\{(\alpha_i,\beta_i)\}=\bigoplus_{i=1}^{r-1}\Hom(\C^i,\C^{i+1})\oplus \Hom(\C^{i+1},\C^i)$$ by the natural action of $\widetilde{H}=\prod_{i=1}^{n-1}\Un(i)$, to obtain the nilpotent cone inside $\sln(n)\cong\sln(n)^*$.  When one takes the hyperk\"ahler quotient by $H=\prod_{i=1}^{r-1}\SU(i)$, however, one obtains an affine variety containing $T^*(\SL(n)/U)$ as a dense open subvariety of codimension at least two (Theorem 7.18 in \cite{DKS}).  As a result, $k[\SL(n)\times \un^{\circ}]^U$ is finitely generated, and $$M\hkq H=\Spec\left(k[\SL(n)\times \un^{\circ}]^U\right)=\left(\SL(n)\times\un^{\circ}\right)\symp U,$$ where the quotient is taken in the sense of non-reductive GIT.
 
 \medskip
 The finite generation of $k[G\times \un^{\circ}]^U$ for general reductive $G$ was established by Ginzburg and Riche (Lemma 3.6.2 in \cite{GinzburgRiche}), and further studied by Ginzburg and Kazhdan in \cite{ginzburg-kazhdan.unpubl,ginzburg-kazhdan.2022}.  We shall refer to the resulting affine space, as a singular complex symplectic variety, as the (right) universal implosion $Q_G=(G \times \un^{\circ})\symp U$.  Remark that $Q_G$ is always normal, following Proposition 1 in \cite{vinberg.popov-1972}.

 \medskip 
 A key ingredient in the work of Ginzburg-Riche and Ginzburg-\linebreak Kazhdan is an action of the Weyl group $W_G$ of $G$ on $Q_G$.  In the real symplectic or algebraic case, we have an action of $K\times T$ (or its complexification) on $T^*K_{\impl}\cong G\symp U$, but in the complex symplectic case $Q_G$ has an action of $G\times(W_G \ltimes T_{\C})$.  From the point of view of the quiver-construction in type $A_n$ of \cite{DKS}, this action of $W_G$ was constructed recently by Wang in \cite{wang.2021}.  Ginzburg-Kazhdan also conjectured that $Q_G$ has symplectic singularities (in the sense of \cite{beauville.2000}). This
 was proved for type $A_n$ using the quiver construction by Jia \cite{jia.2021,jia.2022}, and in the general case by Gannon \cite{gannon.2023}.

\medskip 
Finally, Ginzburg and Kazhdan establish an affine embedding of $Q_G$
(Corollary 7.2.3 in \cite{ginzburg-kazhdan.unpubl}, which follows from the proof of Lemma 3.6.2 in \cite{GinzburgRiche}):

\begin{proposition}
There is a $W_G$-equivariant closed embedding of affine varieties
\begin{equation}\label{GK-embed} Q_G\rightarrow \left(\g^*\oplus \tf_\C^*\right)\times \prod_{w\in W_G} G\symp U.\end{equation}
\end{proposition}
Here the component going to $\g^*\oplus \tf_{\C}^*$ is the moment map for the $G\times T_{\C}$-action, and the component going to each copy of $G\symp U$ is given by the composition of the action of $w\in W_G$ on $Q_G$, and the map ${Q_G\rightarrow G\symp U}$ that is the affinization of the natural projection $T^*(G/U)\rightarrow G/U$. The $W_G$-action on the target is the coadjoint action on $\tf_\C$, and the permutation of the factors of $\prod_{w\in W_G} G\symp U$, while $\g$ is left invariant.  One can combine (\ref{GK-embed}) with (\ref{G/Uembed}) for each $G\symp U$ factor to obtain an embedding into affine space that provides the counterpart to $(\ref{G/Uembed})$

\section{A hypertoric variety related to $Q_G$}
Inside each universal symplectic implosion $(T^*K)_{\impl}\cong G\symp U$ one can embed an affine toric variety $X_G$, preserved by the action of $T_{\C}$, whose image under the $T$-moment map is $\tf^*_+$, and whose $K$-sweep is all of $(T^*K)_{\impl}$, see $[\S 7]$ of \cite{GJS} and the appendix to \cite{HMM}.  In \cite{DKS-Seshadri}, an analogue of this was shown for the universal hyperk\"ahler implosion for $\SL(n)$: there exists a morphism of the affine hypertoric variety $Y_G$, associated with the hyperplane arrangement given by the Weyl chambers for $G$, to $Q_G$, which is generically an embedding.  We generalise this here to the case of arbitrary reductive groups.

\medskip We briefly recall the construction of hypertoric varieties, following \cite{BD,proudfoot.2008}.  For our purpose, it suffices to restrict to affine ones.  These are determined by an arrangement $\mathcal{A}$ of $N$ distinct hyperplanes $H_1, 
\dots, H_N$ in $\tf^*$, all containing the origin.   For each hyperplane, make a choice of a minimal normal vector $\alpha_i\in  \Hom(\mathbb{G}_m, T_{\C})\subset\tf$.  
These choices determine a group homomorphism $\mathbb{G}_m^N\rightarrow T_\C$, which we require to be surjective.  It has kernel $L\triangleleft\mathbb{G}_m^N$, and if we let $\mathbb{G}_m^N$ act diagonally on $\C^N$, and via co-tangent lift on $T^*\C^N\cong \C^{2N}$, then the hypertoric variety is constructed as the complex symplectic reduction of $\C^{2N}$ by $L$: $$\mathfrak{M}(\mathcal{A})=\C^{2N}\hkq L.$$ Up to isomorphism this is independent of the choices of the $\alpha_i$.  We will denote elements of $\C^{2N}$ as $N$-tuples of pairs $(a_i,b_i)$, and will denote the equivalence class of such tuples representing elements in $\mathfrak{M}(\mathcal{A})$ as $[\underline{a},\underline{b}]$.

\medskip In the case of $\mathfrak{M}(\mathcal{A})=Y_G$, $N$ is half the number of roots of the associated root system of $G$.  It will be most convenient in this case if the choice of the normal vectors $\alpha_i$ is induced by the choice of a set of simple positive roots, such that the positive Weyl chamber $\tf^*_+$ is cut out by the equations $\alpha_i(x)\geq 0, \forall i$. As could be expected by the symmetry of the hyperplane arrangement for $Y_G$, this hypertoric variety also has an action of $W_G\ltimes T_{\C}$ (which the toric variety $X_G$ does not have).  This action plays a key role in the construction of the desired map $Y_G\rightarrow Q_G$, even though, surprisingly, this map will not be $W_G$-equivariant.

\begin{lemma}\label{Waction} The $T_\C$-action on the affine hypertoric variety $Y_G$ natural\-ly extends to a $W_G\ltimes T_\C$-action (where $W_G$ is the Weyl group of $G$) such that all moment maps for the action of $T$ and $T_\C$ on $Y_G$ are $W_G$-equivariant.
\end{lemma}
This fact, and the proof below, are valid for any affine hypertoric variety determined by a hyperplane arrangement acted on by a finite group.
\begin{proof} 
The action of the Weyl group on $\tf^*$ permutes the hyperplanes in the arrangement, so we get a homomorphism $\sigma: W_G\rightarrow S_N$.  Moreover, for each $w\in W_G$ and $i\in\{1,\dots,N\}$ we get a sign $s_{w,i}\in\{\pm 1\}$ (a cocycle for the permutation action), determined by $w(\alpha_i)=s_{w,i}\alpha_{\sigma(w)(i)}$.  We obtain an action $\phi$ of $W_G$ on $\mathbb{G}_m^N$, given by $$\phi(w)(t_1, \dots, t_N)=(t_{\sigma(w)(1)}^{s_{w,1}}, \dots, t_{\sigma(w)(N)}^{s_{w, N}}),$$ which makes the morphism $\mathbb{G}_m^N\rightarrow T_\C$ (determined by the $\alpha_i$) $W_G$-equivariant.  As a result, $\phi$ preserves $L$.

\medskip Finally, we can extend the action of $\mathbb{G}_m^N$ on $\C^{2N}$ to an action of $W_G\ltimes \mathbb{G}_m^N$ (note that this does not work for the action of $\mathbb{G}_m^N$ on $\C^N$). Denoting elements of $\C^{2N}$ as $N$-tuples of pairs $(a_i,b_i)\in \C^2$, then each $w\in W_G$ permutes the pairs via $\sigma$, and moreover switches $(a_{\sigma(w)(i)},b_{\sigma(w)(i)})$ to $(b_{\sigma(w)(i)},a_{\sigma(w)(i)}))$ if $s_{w,i}=-1$.

\medskip We immediately get that, after taking the complex symplectic reduc\-tion of $\C^{2N}$ by $L$ (at zero-level of the moment maps), the action of $W_G\ltimes \mathbb{G}_m^N$ descends to an action of $\left(W_G\ltimes \mathbb{G}_m^N\right)\Big/ L \cong W_G\ltimes T_\C$ on the hypertoric variety.  Moreover, all the moment maps for the action of $T$ on $Y_G$ are $W_G$-equivariant. \
\end{proof}

\medskip We now want to define a morphism from $Y_G$ into the target of (\ref{GK-embed}), whose image is contained in the image of $Q_G$.  The components of this map into both $\g^*$ and $\tf_{\C}^*$ are just given by the moment map for the action of $T_{\C}$ on $Y_G$ -- remark that such a map from $Y_G$ into $\g^*$ is not $W_G$-invariant, whereas the component into $\g^*$ of (\ref{GK-embed}) is $W_G$-invariant (this at the end will result in the map $Y_G\rightarrow Q_G$ not being $W_G$-equivariant).

\medskip The components of the map into the various copies of $G\symp U$ will land in $X_G\subset G\symp U$ for each $w$.  If $Y_G$ were to be a smooth hypertoric variety, it would be covered by cotangent bundles to the various components of its extended core (see \cite{proudfoot.2008}, Remark 2.1.6).  These components are Lagrangian subvarieties, each isomorphic to $X_G$, and permuted by the action of $W_G$ on $Y_G$.  Under the real moment map their images are given by the (closed) Weyl chambers in $\tf^*$.  The various projections $T^*X_G\rightarrow X_G$ would extend to give maps $Y_G\rightarrow X_G\subset G\symp U$, and these would give the remaining components of the morphism from $Y_G$ into the target of (\ref{GK-embed}).

\medskip However, the variety $Y_G$ will be singular in general, and therefore this argument needs a bit more care.  We can use a smooth stratification of affine hypertoric varieties introduced in \cite{proudfoot.webster-2007}, \S 2.  Given a hyperplane arrangement $\mathcal{A}$,\ 
and a subset $F\subset \{1, \dots, N\}$, the affine subspace $\cap_{i \in F} H_i$ is denoted $H_F$.  If $F=\{i| H_i\subset H_F\}$, $F$ is called a \emph{flat}.  For every flat $F$,  $\mathcal{A}^F$ is defined to be the hyperplane arrangement $\{H_i\cap H_F|i\notin F\}$ in the affine space $H_F$.  Similarly,  $\mathcal{A}_F$ is defined to be the hyperplane arrangement $\{H_i/H_F|i\in F\}$ in $\tf^*/H_F$.  We have that the $\mathfrak{M}(\mathcal{A}^F)$ are all symplectic subvarieties of $\mathfrak{M}(\mathcal{A}^{\emptyset})=\mathfrak{M}(\mathcal{A})$. Finally, Proudfoot and Webster define $\overset{\circ}{\mathfrak{M}}(\mathcal{A}^F)$ to be 
$$\overset{\circ}{\mathfrak{M}}(\mathcal{A}^F)=\Big\{[\underline{a},\underline{b}]\in \mathfrak{M}(\mathcal{A})\ \Big|\ a_i=b_i=0\ \Leftrightarrow\ i\in F\Big\}$$ (this is always a smooth subvariety).
It is shown in Lemma 2.4 of \cite{proudfoot.webster-2007} that the decomposition $$\mathfrak{M}(\mathcal{A})=\bigsqcup_{F \ \text{flat}}\overset{\circ}{\mathfrak{M}}(\mathcal{A}^F)$$ is a stratification, with a normal slice to each $\mathfrak{M}(\mathcal{A}^F)$ provided by $\mathfrak{M}(\mathcal{A}_F)$.
Since all of these strata are even-dimensional, and two-dimen\-sional hypertoric varieties (when hyperplanes are not repeated) are always smooth, this immediately gives us
\begin{lemma} The singular locus of an affine hypertoric variety will have (complex) codimension at least four.
\end{lemma}
 In particular we can focus on $\mathfrak{M}(\mathcal{A})^g=\overset{\circ}{\mathfrak{M}}(\mathcal{A})\cup \bigcup_{i\in \{1, \dots, N\} }\overset{\circ}{\mathfrak{M}}(\mathcal{A}^{\{H_i\}})$.

\medskip   Furthermore, for each $V\subset \{1, \dots, N\}$, we can look at the cone $\mathcal{C}_V$ cut out by $\alpha_i(x)\geq 0$ if $i\in V$ and $\alpha_i(x)\leq 0$ if $i\notin V$.  If this cone is top-dimensional, we say $V$ is \emph{broad}, and we define $$\mathcal{X}(V)=\Big\{\ [\underline{a}, \underline{b}]\ \Big|\  b_i=0 \ \text{if}\ i\in V\ \text{and}\ a_j=0\ \text{if}\ j\notin V\Big\}.$$ 
The $\mathcal{X}(V)$, for all broad $V$, are the components of the \emph{extended core}.  They are Lagrangian subvarieties of $Y_G$, that are toric for the action of $T_{\C}$, with at worst finite quotient singularities, and whose image under the real moment map is given by $\mathcal{C}_V$.  They will intersect the singular locus of $Y_G$, but we can restrict our attention to what happens in $\mathfrak{M}(\mathcal{A})^g$.  We put $\mathcal{X}(V)^g=\mathcal{X}(V)\cap \mathfrak{M}(\mathcal{A})^g$, which is the subvariety of the toric variety $\mathcal{X}(V)$ consisting of the open orbit and the codimension-one orbits.
\begin{lemma}\label{coverandcodim}
The cotangent bundles $T^*\mathcal{X}(V)^g$, for all broad $V$, are symplectic subvarieties that cover $ \mathfrak{M}(\mathcal{A})^g$.  Each of these $T^*\mathcal{X}(V)^g$ has (complex) codimension at least two in $\mathfrak{M}(\mathcal{A})$.
\end{lemma}
\begin{proof} It suffices to remark that $$T^*\mathcal{X}(V)^g=\Big\{ [\underline{a}, \underline{b}]\, \Big|\, \text{at most one of the}\ a_i, i\in V,\ \text{and}\ b_j, j\notin V,\, \text{is zero}
\Big\}.$$
\end{proof}

\medskip We now return our focus to the particular hypertoric variety $\mathfrak{M}(\mathcal{A})=Y_G$.  Here the cotangent bundles $T^*\mathcal{X}(V)^g\subset Y_G$, for broad $V$, are permuted by the $W_G$-action.  Moreover, for $\Omega=\{1,
\dots ,N\}$, we have $\mathcal{X}(\Omega)=X_G\subset Y_G$.  

\medskip The natural projection $T^*\mathcal{X}(\Omega)^g\rightarrow X_G$ extends to a morphism ${Y_G\rightarrow X_G}$ by the algebraic version of Hartogs' theorem, since affine hypertoric varieties are always normal (see see \cite{bellami-kuwabara.2012}, Proof of 4.11, or \cite{spenko-vandenberg.2021}, \S 4) and Lemma \ref{coverandcodim}.  Precomposing this map with the action of $w\in W_G$ on $Y_G$ gives us the desired components into the other copies of $X_G$, to 
finally obtain a morphism of affine varieties
\begin{equation}\label{hypertoricemb}Y_G\rightarrow \left(\g^*\oplus \tf_{\C}^*\right)\times \prod_{w\in W_G}X_G. \end{equation}

\begin{lemma}
The map (\ref{hypertoricemb}) restricted to $\overset{\circ}{Y_G}\subset Y_G$ is an embedding.
\end{lemma}
Remark that the locus on which the map $Y_G\rightarrow Q_G$ is an embedding is the same as the one described in \cite{DKS-Seshadri} using the quiver construction for $Q_{\SL(n,\C)}$.
\begin{proof}
Similar to Lemma \ref{coverandcodim}, we have that $\overset{\circ}{Y_G}$ is covered by\linebreak $T^*\left(\mathcal{X}(V)\cap \overset{\circ}{Y_G}\right)$, for broad $V$.  Each of the latter, in turn, is $$T^*\left(\mathcal{X}(V)\cap \overset{\circ}{Y_G}\right)\ =\ \Big\{\ [\underline{a}, \underline{b}]\ \Big|\ a_i\neq 0\ \text{if}\  i\in V\ \text{and}\ b_j\neq 0\ \text{if}\ j\notin V \Big\},$$ which is isomorphic to $T^*T_{\C}$.  It therefore suffices to remark that, for any Lie group $H$, the map $T^*H\rightarrow H\times \mathfrak{h}^*$, given by the natural projection onto $H$, and the moment map for the cotangent lift of the action $H\actson H$ by multiplication by inverses on the right, is an isomorphism.  Hence  the map from $T^*\left(\mathcal{X}(\Omega)\cap \overset{\circ}{Y_G}\right)$ to ${\left(T_\C\subset X_G\subset G\symp U\right)\times \tf_{\C}^*}$ is an embedding.  As $W_G$ permutes the broad $V$, the full map $(\ref{hypertoricemb})$ is an embedding on $\overset{\circ}{Y_G}$.
\end{proof}
\begin{proposition}
The image of (\ref{hypertoricemb}) is contained in the image of (\ref{GK-embed}), hence we have a morphism $Y_G\rightarrow Q_G$ that is an embedding on $\overset{\circ}{Y_G}$. 
\end{proposition}
\begin{proof}As (\ref{GK-embed}) is a closed embedding, it suffices to show this for the dense open subvariety $T^*\left(\mathcal{X}(\Omega)\cap\overset{\circ}{Y_G}\right)\cong T^*T_\C$.  We can embed $T^*T_\C$ naturally into $T^*(G/U)\subset Q_G$, by $$(t,\alpha)\in T_\C\times\tf_\C^*\mapsto \overline{(t,\alpha)}\in T^*(G/U)\cong (G\times \mathfrak{u}^{\circ})/U.$$ (We use here that $\tf_C^*$ naturally is a subspace of $\mathfrak{u}^\circ$.)    As it is straight\-forward to verify that the images of $T^*T_{\C}$ under (\ref{GK-embed}) and (\ref{hypertoricemb}) are identical, we are done.
\end{proof}
\section{Implosion and the Moore-Tachikawa category}
\label{MT}
In this section, we discuss the Moore-Tachikawa category $HS$ and explain how it gives a useful organising principle for implosion and contraction.

\medskip The objects of $HS$, as introduced in \cite{MT}, are complex semisimple groups and the morphisms
from $G_1$ to $G_2$ are complex symplectic varieties with linearised, Hamiltonian,
$G_1 \times G_2$-actions. In the setup of \cite{MT} one also assumes a 
commuting circle action acting on the complex symplectic form with
weight 2, but we do not need that for our discussion, and we will also allow for objects to be complex reductive groups.

\medskip
The composition $Y \circ X$ of morphisms $X \in {\Hom}_{HS}(G_1, G_2)$
and $Y \in {\Hom}_{HS}(G_2, G_3)$
is given by the complex symplectic quotient
\begin{equation} \label{quotient}
Y \circ X = (X \times Y) \hkq \Delta G_2
\end{equation}
where $\Delta G_2$ denotes the diagonally embedded $G_2$ in the
symmetry group $G_1 \times (G_2 \times G_2) \times G_3$ of $X \times Y$.
This quotient is complex symplectic with
residual $G_1 \times G_3$-action so lies in $\Hom_{HS}(G_1, G_3)$
as required.
If $K$ is a compact Lie group with $G=K_{\C}$, then the Kronheimer space $T^*G$
\cite{Kronheimer:cotangent} is complex symplectic
with $G \times G$-action and defines the identity element
in $\Hom_{HS}(G,G)$, due to the complex analogue of (\ref{trivial}).

\medskip More generally, we could consider a Moore-Tachikawa category in any setting where one has a reduction theory for group actions on spaces of some kind, provided that reduction by commuting actions commutes, and there exists an identity.  This is the case for Hamiltonian actions of compact groups on symplectic manifolds and stratified spaces, and linearised actions of reductive groups on varieties projective over an affine.  As mentioned above, an identity is missing in the case of K\"ahler reduction of K\"ahler manifolds by compact groups, or hyper-K\"ahler reduction -- in those cases one only has a semi-category.

\medskip In \cite{MT} it was conjectured that a choice of group $G$ would give a functor
$\eta_G : Bo_2 \rightarrow H$, where $Bo_2$ is the category of $2$-bordism,
mapping the cylinder to $T^*G$. The Riemann sphere $\C \PP^1$ maps to
the BFM space (or universal centraliser) \cite{BFM.2005}, while the once-punctured sphere maps to
the product of $G$ and the Slodowy slice. This is actually a symmetric monoidal functor, where the monoidal structure on HS is given by the product of groups and 
of spaces, and the dual of a group is just itself.
The functor thus defines a 2-dimensional TQFT.
Recently Ginzburg and Kazhdan  \cite{ginzburg-kazhdan.unpubl,ginzburg-kazhdan.2022} and Bielawski \cite{bielawski.2023} have
defined the image under $\eta_G$ of the three-times punctured sphere --- together with the above
data this determines the functor.
 Physically, $\eta_G(C)$ represents the Higgs branch of a $N=2$ supersymmetric theory 
associated to a punctured Riemann surface $C$ with defects located at the punctures,
associated to homomorphisms
$\su(2) \rightarrow \g$.

\medskip
As mentioned, in the original Moore-Tachikawa discussion the
groups are taken to be semisimple. In what follows we shall find it useful to extend the formalism to more general reductive groups.

\medskip
As discussed in \cite{DS} and \cite{DHK},
the universal hyperk\"ahler implosion for a compact $K$ with maximal torus $T$ may be viewed as an
element of $\Hom_{HS}(G, T_{\C})$. The process of imploding
a complex symplectic manifold $M$ with $G$-action, viewed as an element
of $\Hom_{HS}(1, G)$, to obtain
a manifold $M_{\impl}$ with $T_{\C}$-action,
is now exactly that of composition of morphisms with the universal implosion.
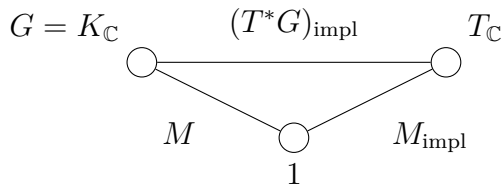
\begin{figure}[H]
  \centering
\begin{tikzpicture}
  \node (ag1) at (-2,-2.5) [gauge, label=above left:{$G=K_\C$}]{};
  \node (ag2) at (2,-2.5) [gauge, label=above right:{$T_\C$}]{};
  \node (ag3) at (0,-3.5) [gauge, label=below:{1}]{};
  \draw (ag1)--(ag2)--(ag3)--(ag1);
  \node at (0,-2.5) [label=above:{$(T^*G)_{\impl}$}]{};
  \node at (-1, -3) [label=below left:{$M$}]{};
  \node at (1, -3) [label=below right:{$M_{\impl}$}]{};
\end{tikzpicture}
\end{figure}
Of course, one could also define a similar picture with compact groups replacing
complex reductive groups, and with (real) symplectic manifolds playing the role
of morphisms. Again, the universal symplectic implosion $(T^*K)_{\impl}$
for a compact $K$ is an element of
$\Hom_{HS}(K,T)$, and implosion of a symplectic $K$-manifold as defined in
\cite{GJS} is just composition
with the universal implosion. Explicitly,
we have
\[
M_{\impl} = (M \times (T^*K)_{\impl}) \symp \Delta K.
\]

\section{Symplectic and algebraic contraction}

Returning to our definition of composition in $HS$, we see that if the middle group
$G_2$ is \emph{Abelian}, then in fact $Y \circ X$ has an action not just
of $G_1 \times G_3$, but of $G_1 \times G_2 \times G_3$, because the
\emph{anti}-diagonally embedded $G_2$ commutes with the diagonal $G_2$
and hence descends to the quotient.
In particular, if we have a space with $G$-action, we may
compose it with the (right) universal implosion (viewed as a morphism from
$G$ to $T_{\C}$) to obtain a space with $T_{\C}$-action, i.e. we form
the implosion. Then we may
compose again with the (left) universal
implosion considered as a morphism in the reverse direction, and get
a space with $G \times T_{\C}$-action, where the $T_{\C}$ factor arises because the middle group $T_{\C}$ in the composition is Abelian.
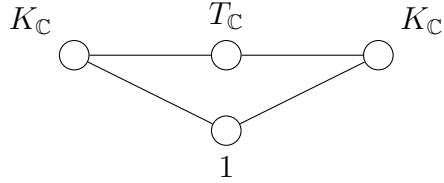
\begin{figure}[H]
\centering
\begin{tikzpicture}
  \node (ag1) at (-2,-2.5) [gauge, label=above left:{$K_\C$}]{};
  \node (ag2) at (0,-2.5) [gauge, label=above:{$T_\C$}]{};
  \node (ag3) at (2,-2.5) [gauge, label=above right:{$K_\C$}]{};
  \node (ag4) at (0,-3.5) [gauge, label=below:{1}]{};
  \draw (ag1)--(ag2)--(ag3)--(ag4)--(ag1);
\end{tikzpicture}
\end{figure}
In the setting of real symplectic manifolds with compact group actions, the resulting space is the symplectic contraction of \cite{HMM}.
The dimension calculation in the real case is as follows, denoting the contraction of $X$ by $X_{\syco}$:
\[
\dim X_{\impl} = \dim X + \dim (T^*K)_{\impl} - 2 \dim K = \dim X + \rank K
- \dim K
\]
\begin{eqnarray*}
  \dim X_{\syco} &=& \dim X_{\impl} + \dim (T^*K)_{\impl} - 2 \dim T \\
  &=& \dim X + \rank K - \dim K + (\dim K + \rank K)
      - 2 \rank K \\
&=& \dim X.
\end{eqnarray*}
So the contraction keeps the dimension the same but enhances the $K$-action to
a $K \times T$-action.

\medskip
As with implosion, the key example is the universal
contraction, i.e. the contraction of $T^*K$ 
with $K \times K \times T$-action.
We recover the contraction of $X$
by taking the product with the universal contraction and reducing by $K$.
In the symplectic case
\[
(T^*K)_{\syco} = ( (T^*K)_{\rm right \; \impl} \times (T^*K)_{\rm left \; \impl}) \symp T.
\]

\smallskip
Recall that the left $K$-action on $T^*K$ is
\[
h. (k,v) \mapsto (hk,v)
\]
with moment map
\[
\mu_L : (k, v) \mapsto -k.v
\]
where $k.v$ denotes the coadjoint action of $k$ on $v$. The right action is
\[
h.(k,v) \mapsto (kh^{-1}, h.v)
\]
with the moment map
\[
\mu_{R} : (k,v) \mapsto v
\]
so that points of $(T^*K)_{\rm right \; \impl}$ are represented by pairs $(k,v)$
with\linebreak $k \in K$ and $v$ in the closed positive Weyl chamber $\bar{C}$.
Points of $(T^*K)_{\rm left \; impl}$ are represented by pairs $(k,v)$ with $k \in K$ and $-k.v \in \bar{C}$. The
involution $(k,v) \mapsto (k^{-1},-k.v)$ intertwines the left and right actions and moment maps, and induces an identification of the corresponding implosions.

\medskip
In \cite{HMM}, a map $\Phi : T^*K \rightarrow (T^*K)_{\syco}$ is defined by
\begin{equation} \label{surj}
(k, v) \mapsto [(kh^{-1}, h.v), (h, v)]
\end{equation}
where $h$ is such that $h.v$ lies in the chosen closed positive Weyl chamber. We
use square brackets to denote the fact we are passing to the quotient by $T$.
Note that $h$ is uniquely defined up an element of   $K_{\sigma}$, where $K_{\sigma}$ is the stabiliser corresponding to the
fact $\sigma$ of $\bar{C}$ containing $h.v$. But we collapse by $[K_\sigma, K_\sigma]$ in forming the implosion, and by $T$ in taking the symplectic quotient to
form $(T^*K)_{\syco}$, so the map (\ref{surj}) is well-defined.  

\medskip
This map is surjective, as proven in \cite{HMM}, because given any\linebreak
$[(k,w), (h,v)]$ $ \in (T^*K)_{\syco}$, we have $w = h.v \in \bar{C}$ by the
0-level set condition  for the $T$-moment map.
Now we see that $$\Phi :(kh,v) \mapsto [(k, h.v),(h,v)] =[(k,w),(h,v)].$$
As a result, we can understand $(T^*K)_{\syco}$ or, in fact, any $M_{\syco}$ set-theoreti\-cally as $M_{\syco}=M/ \thicksim$, where $$x\sim y \Leftrightarrow \mu(x)=\mu(y) \ \text{and}\ y \in [K_{\mu(x)}, K_{\mu(x)}].x,$$ where $K_{\mu(x)}$ is the stabiliser in $K$ of $\mu(x)$ for the coadjoint action.

\medskip 
The algebraic counterpart of symplectic contraction is given by\linebreak Popov's notion of horospherical contraction \cite{popov.1986} for varieties with an action of a reductive group.  
This was initially defined for an affine $G$-variety $X$ as follows: the ring $k[X]$ is equipped with a poly-filtration obtained from the decomposition into isotypical components $k[X]_{(\lambda)}$ for $G$ (where $\lambda$ runs over the dominant weights).  By choosing a regular dominant co-weight $\rho\in \Hom(\Un(1),T)\cap \tf_+$, we can make this into an $\mathbb{N}$-filtered ring by putting $$k[X]_{\leq n}=\bigoplus_{\rho(\lambda)\leq n} k[X]_{(\lambda)}.$$  The horospherical contraction of $X$ is now determined by the associated graded of this algebra: $$X_{\hoco}=\Spec(\gr(k[X])),$$ which is independent of the choice of $\rho$.   (This definition naturally extends to a broader setting of varieties that are projective over an affine.)  The Rees construction moreover naturally creates a $\mathbb{G}_m$-equi\-variant flat family ${\pi:\chi\rightarrow \mathbb{A}^1}$, the generic fiber of which is isomorphic to $X$, and the fiber over $0$ is given by $X_{\hoco}$.

\medskip
The
universal horospherical contraction is the asymptotic semigroup $\As(G)$, as defined by Vinberg \cite{vinberg.1995b}.  This is the central fiber in a degeneration of $G$ known as the Vinberg monoid $S_G$ \cite{vinberg.1995a}.  This is obtained via the Rees construction as above, but without switching to an $\mathbb{N}$-filtration by a choice of $\rho$ as above -- this results in a flat fibration $\pi_G:S_G\rightarrow \mathbb{A}_G$ over a higher dimensional base, which is the (smooth) affine toric variety for the torus $T/Z_K$ (where $Z_K$ is the center of $K$) given by the cone spanned by the positive simple roots.  Vinberg shows that $S_G$ is a reductive monoid, with group of units $S_G^{\times}\cong (G\times T_{\C})/Z_G$.  A choice of a regular dominant $\rho$ as before gives rise to a morphism of monoids $\mathbb{A}^1\rightarrow \mathbb{A}_G$, and the flat family over $\mathbb{A}^1$ induced by this choice is the base-change of $\pi:S_G\rightarrow \mathbb{A}_G$ by this morphism.

\medskip
With a suitable (but not unique) choice of K\"ahler structure, we can identify $X_{\syco}$ and $X_{\hoco}$, see \S 5 and 6 in \cite{HMM}. This is done, following Harada-Kaveh \cite{harada-kaveh.2015}, by means of the gradient-Hamiltonian vector field \begin{equation}\label{grad-hamil}V_{\pi}=-\frac{\nabla \Re(\pi)}{\norm{\nabla\Re(\pi)}^2}\end{equation} on the smooth locus of the total space of the degeneration $\chi$, whose flow induces a continuous surjective map $X\rightarrow X_{\hoco}$ that generically is a symplectomorphism.  In general the results of \cite{harada-kaveh.2015}
just guarantee existence of this map, but in the case of a degeneration corresponding to a horospherical contraction, as we are considering, it can be made explicit.  What enables this is a reduction to the universal case\linebreak ${\pi_G:S_G\rightarrow \mathbb{A}_G}$, where the monoid structure of $S_G$, and Lie group decom\-positions for $S_G^{\times}$, make the flow for $V_{\pi}$ tractable.  This allows the identification of $\As(G)$ and $(T^*K)_{\syco}$, which in turn implies the same for $X_{\hoco}$ and $X_{\syco}$.

\section{Complex symplectic contraction}
\label{complex}
Given the interpretation for the contraction in the real symplectic and algebraic situations in terms of implosion, it is now natural to 
take the product of the complex symplectic implosion of a space $M$ and the left implosion of $T^*G$, and reduce this by the diagonal complex torus action.  We shall refer to this as the complex symplectic contraction, and denote it as $M_{\cosyco}$.

\medskip
The right implosion is the complex symplectic quotient, in the GIT sense,
of $T^*G = G \times \g^*$ by the maximal unipotent subgroup $U$ of $G$, that is,
 the GIT quotient $Q^{R} = (G \times \un^\circ) \symp U$. This can also
 be viewed as $(G \times \bmf) \symp U$ where
 $\bmf$ denotes the Borel algebra.
 We could also form the left implosion
 \[
Q^{L} = \Big\{ (g, v) \in G \times \g\ \Big|\ g.v \in \bmf \Big\} \Bigsymp U
\]
where $U$ now acts on the left, i.e. by translation in the $G$-factor. As before, the
involution $(g,v) \mapsto (g^{-1},-g.v)$
gives an identification
between left and right implosions.

\begin{definition}
We now define the universal complex symplectic contraction to be the complex-symplectic quotient
by the diagonal complex torus $T_{\C}$ of the product of right and left implosions:
\[
(T^*G)_{\cosyco} =(Q^R \times Q^L) \hkq_0 T_{\C}.
\]
\end{definition}
Now $M_{\cosyco}$ can be obtained by taking the product of $M$ with the universal example
$(T^*G)_{\cosyco}$ and reducing by $G$.

\medskip
The above discussion shows the following.
\begin{proposition}
$(T^*G)_{\cosyco}$
has the same dimension as $T^*G$ but the 
complex-symplectic $G \times G$-action is enhanced to a $G \times T_{\C} \times G$-action. 

Similarly $M_{\cosyco}$ has the same dimension as $M$ but its $G$-action is enhanced to a $G \times T_{\C}$-action.
\end{proposition}

\begin{remark}
The complex-symplectic reduction of $(T^*G)_{\cosyco}$ by the $T_{\C}$-action is the 
same as the product of two copies of the reduction of the implosion by $T_{\C}$ -- in particular if we reduce at level zero we obtain the product of two copies of the nilpotent cone.
\end{remark}
\medskip
Let us discuss some properties of the
universal example.
How could we mimic the map (\ref{surj}) from the symplectic case?
We want a map $T^*G \rightarrow (T^*G)_{\cosyco}= (Q^R \times Q^L) \hkq_0 T_\C$.
Given $(g,v) \in T^*G$, we can choose $h \in G$ with $h.v \in \bmf$
as every element of $\g$ is conjugate via the adjoint action to an element
of the fixed Borel algebra $\bmf$. So we could try to define
\begin{equation} \label{Psi}
  \Psi : (g,v) \mapsto [(gh^{-1}, h.v), (h,v)]
\end{equation}
and the right-hand side then lies in the 0-level set of the $T_\C$-action as required.

\medskip 
However, the problem here is that the condition $h.v \in \bmf$ does not
specify $h$ up to an element of the Borel subgroup $B$ in general.
It is therefore natural to replace the domain $T^*G$ by
$G \times \tilde{\g}$ where $\tilde{\g}$ is the Grothendieck simultaneous
resolution. That is, we consider
\[
G \times \tilde{\g} = \Big\{ (g,v, \bmf_1)\ \Big|\ g \in G, v \in \bmf_1 : \; \bmf_1 \in
{\mathcal B} \Big\}
\]
where $\mathcal B = G/B$ is the variety of Borel algebras in $\g$.

\begin{proposition}
The formula (\ref{Psi}) above defines a map
from
$G \times \tilde{\g}$ to the complex-symplectic contraction $(T^*G)_{\cosyco}$.
\end{proposition}

\begin{proof}
Given $(g,v\bmf_1) \in G \times \tilde{\g}$, there 
is now $h \in G$ with $h. \bmf_1 = \bmf$ (so $h.v \in \bmf$),
and $h$ \emph{is} unique
up to an element of $B$. Therefore the right-hand side
of (\ref{Psi}) is well-defined, as the terms in 
round brackets are elements of the quotients by 
the maximal unipotent $U$, and the square bracket denotes the equivalence class 
under the $T_\C$-action.
\end{proof}

\begin{remark}
We could generalise this construction to any $X$ with Hamiltonian $G_\C$-action,
and form
\[
\Big\{ (x, \bmf_1 ) \ \Big|\ \mu(x) \in \bmf_1 \big\} \subset X \times \mathcal B.
\]
If $X = T^*G$ then the moment map is projection to the second factor,
and identifying $\g^*$ with $\g$ we obtain $G \times \tilde{\g}$ as above. 
\end{remark}

\medskip
Recall from \cite{DKS} that every orbit of $U$ in the open set
\[
\Big\{ (g, v) \in G \times \bmf \ \Big|\  \tf-{\rm component \; of \;} v \; {\rm is} \;
\in \tf_{\rm reg} \Big\}
\]
(where $\tf_{\rm reg}$ consists of the regular elements in $\tf$) has a unique representa\-tive with $v \in \tf_{\rm reg}$.
We therefore have a set $Q^R_\circ = G \times \tf_{\rm reg}$ contained in $Q^R$, and of the same dimension as $Q^R$, as well as an analogous subset
$Q^{L}_{\circ} \subset Q^L$.
We may consider the locus $Q^{R}_\circ \times Q^{L}_{\circ}$ and its projection $(T^*G)_{\cosyco,\circ}$
to the symplectic quotient by $T_{\C}$. Points in this locus are represented
by pairs
\[
( (g_1, v), (g_2, w) ) \; : \; g_1, g_2 \in G \; : \; v, \;
g_2 . w \in \tf_{\rm reg}
\]
subject to the zero-level set condition for the $T_\C$-moment map
\[
v - g_2. w =0.
\]
\begin{proposition}
The locus $(T^*G)_{\cosyco,\circ}$ is contained in the image
of the map $\Psi$.
\end{proposition}
\begin{proof}
For such points, we can proceed as in the real symplectic case,
for now $v=g_2.w \in \tf \subset \bmf$ and $(g_1 g_2, w, \bmf_1) \in
G \times \tilde{\g}$ maps to
\[
[ (g_1, g_2.w), (g_2, w) ] = [ (g_1,v), (g_2,w) ],
\]
where we choose $\bmf_1$ to be  the Borel algebra containing $w$
such that $g_2. \bmf_1 = \bmf$. 
So we have surjectivity onto the locus represented by ${Q^{R}_\circ \times
Q^{L}_\circ}$.
\end{proof}

\begin{remark}
    We could also see this by recalling that we have an identification 
    $G \times_B \bmf \cong \tilde{\g}$ given by
    \[
    (g^{\prime},x) \mapsto (g^{\prime}.x,g^{\prime}B/B).
    \]
So $\Psi$ can be viewed as a map from
$G \times G \times_B \bmf$ to the contraction
given by the composition
\[
(g,g^{\prime},x) \mapsto (g,g^{\prime}.x,
g^{\prime}B/B)
\mapsto [(g g^{\prime},x), ((g^{\prime})^{-1}, g^{\prime}.x)]
\]
as we can take $h= (g^{\prime})^{-1}$ in the formula (\ref{Psi}). 
\end{remark}
\begin{remark}
As the complex symplectic implosion has an action
of $W_G\ltimes T_{\C}$, and $$\Delta T_{\C}\mathrel{\unlhd}\Delta W_G\ltimes (T_{\C}\times T_{\C})\subset (W_G\ltimes T_{\C})\times (W_G\ltimes T_{\C}),$$ 
we see that the contraction inherits an action of the
diagonally embedded Weyl group also.  In particular the universal complex symplectic contraction $T^*G_{\cosyco}$ has an action of $G\times G\times (W_G\ltimes T_{\C})$.
\end{remark}
\begin{proposition}
The real symplectic contraction sits naturally inside the complex-symplectic contraction.
\end{proposition}
\begin{proof}
 Recall from \S 4 of \cite{DKS-Arb}  that the real symplectic implosion $G \symp U$ arises as the fixed point locus  of the $\C^*$-action on the hyperk\"ahler implosion induced from scaling in the $\un^{\circ}$ factor.
 The $T_{\C}$-moment map is equivariant for this action, so the product 
 of left and right real symplectic implosions lies in the zero locus of the $T_{\C}$-moment map in
 $Q^R \times Q^L$, and taking $T_{\C}$-quotients the statement follows. 
\end{proof}
\section{Examples}
\label{examples}

Let us consider the simplest case, where $K = \SU(2)$.

\medskip
We first look at the universal symplectic contraction. The universal symplectic implosion
is $\C^2$ with the flat K\"ahler structure and $U(2)$-action, so the contraction
is the symplectic quotient
\[
(\C^2 \times \C^2) \symp U(1),
\]
where $U(1)$ acts on the copies of $\C^2$ by scalar multiplication by
$e^{i \theta}$ and $e^{-i \theta}$ respectively.
Equivalently, this is the GIT quotient by $\C^*$ with action
\[
  \left( \begin{array}{c}
           v_1 \\
           v_2
         \end{array} \right) \mapsto
       \left( \begin{array}{c}
           t v_1 \\
           t v_2
              \end{array} \right), \;\;\; : \;\;\;
(w_1 \;\; w_2 ) \mapsto (t^{-1}w_1 \;\; t^{-1}w_2).
\]
The invariants are given by the entries of
\[
\left(\begin{array}{cc}
  X & Y \\
  Z & W
 \end{array}\right) = v \otimes w = (v_i w_j)_{i,j=1}^{2}.
\]
This is the affine variety in $\C^4$ defined by $XW - YZ =0$, which
is the symplectic contraction for $\SU(2)$.  Indeed, in this case the Vinberg monoid $S_{\SL(2)}$ is just the space of $2\times 2$ matrices, fibering over $\C$ by the determinant, so that the central fiber, the asymptotic semigroup,  consists of the singular matrices.  

\medskip To understand the flow for the vectorfield (\ref{grad-hamil}), where $\pi=\det$, we can argue as follows (cf.\cite{HMM}). Since both $\pi$ and the K\"ahler metric on $S_{\SL(2)}$ are invariant under left and right multiplication by matrices in $\SU(2)$, the flow for $V_{\pi}$ is equivariant for these actions as well.  Now consider the $KAK$-type Cartan decomposition for $S_{\SL(2)}^{\times}=\GL(2,\C)$ ($A$ here consists of the diagonal matrices in $\GL(2,\R)$ with positive entries).  If $x\in \SL(2,\C)$ is contained in $k_1Ak_2$ for some $k_1,k_2\in \SU(2)$, by the equivariance, the entire flow of $x$ under $V_{\pi}$ will be as well.  So it suffices to just look at elements $\begin{pmatrix}x&0\\ 0&x^{-1}\end{pmatrix}\in (A\cap \SL(2,\C))$. Level sets of $\pi$ restricted to $A$ are just given by components of hyperbola $$\left\{\begin{pmatrix}x&0\\ 0&y\end{pmatrix}\ \Big|\ x,y\in\R_+\ \text{with}\ xy=\lambda\right\},$$ 
and hence the flow lines for $V_{\pi}$, which are everywhere orthogonal to them, lie on hyperbola given by $\begin{pmatrix}x&0\\ 0&y\end{pmatrix}$ with $x^2-y^2=c$.  This in turn allows us to determine exactly what the flow for time $t=1$ does to elements in $A\cap \SL(2,\C)$:
$$\begin{pmatrix}x& 0\\ 0 & x^{-1}\end{pmatrix}\mapsto \begin{cases}\begin{pmatrix}\sqrt{x^2-x^{-2}} &0 \\ 0 &0\end{pmatrix} &\text{if}\ \ x\geq 1,\\ \ & \ \\ \begin{pmatrix}0 &0 \\ 0 & \sqrt{x^{-2}-x^{2}}\end{pmatrix} &\text{if}\ \ x\leq 1. \end{cases}$$  We can translate this cleanly to all of $\SL(2,\C)$ using the polar decom\-position as follows:
\begin{equation}\label{examplecontraction}\SL(2,\C)\ni B=UP=U\sqrt{B^*B}\mapsto U\sqrt{B^*B-\gamma \Id},\end{equation} where $\gamma$ is the smallest eigenvalue of $B^*B$.  From this one immediately sees that the moment maps for the actions of $\SU(2)$ given by multiplica\-tion on the left and the right are invariant under this map. Moreover, the stabilisers for these actions remain the same, except when $B^*B=\Id$, i.e. when $B\in\SU(2)$ (all of which get sent to the zero matrix).  So (\ref{examplecontraction}) just collapses $\SU(2)\subset\SL(2.\C)$, which, using the Cartan decomposition, we can think of as the zero section of $T^*\SU(2)\cong \SL(2,\C)$.  This is of course the exactly the same as what happens for $T^*\SU(2)_{\syco}$.

\medskip
In the hyperk\"ahler case, the universal implosion is $$Q_{\SL(2,\C)}=\HH^2= \Hom(\C, \C^2) \oplus
\Hom(\C^2,\C)$$ with $\SU(2)$-action
$(\alpha, \beta) \mapsto (g \alpha,\beta g^{-1})$
and $U(1)$-action\linebreak $(\alpha,\beta) \mapsto (e^{-i \theta} \alpha,e^{i \theta}\beta)$.

To see the action of the Weyl group $W_{\SL(2,\C)}=\{1,\gamma\}$, it is easiest to use the identification $\SL(2,\C)=Sp(2,\C)$.  If we take $J=\left(\begin{array}{cc} 0 &1\\ -1 &0\end{array}\right)$, then for any $g\in \SL(2,\C)$ we have $g^TJ=Jg^{-1}$.  We can now put $$\gamma.(\alpha,\beta)=\left((\beta J)^T, (J\alpha)^T\right),$$ i.e. $$\gamma.\left(\left(\begin{array}{c} \alpha_1\\ \alpha_2\end{array}\right), (\beta_1\ \beta_2)\right)=\left(\left(\begin{array}{c} -\beta_2 \\ \beta_1\end{array}\right), (\alpha_2\ -\alpha_1)\right).$$
This all indeed combines to given an action of $\SL(2,\C)\times \left(W_{\SL(2,\C)}\ltimes \C^{*}\right)$ on $\mathbb{H}^2$ (where $\gamma.\lambda=\lambda^{-1}$ for $\lambda\in\C^*)$. Note that the $\SL(2,\C)$-moment map $(\alpha \beta)_0$ (ie. the tracefree part of $\alpha \beta$) is Weyl-invariant, while the $\C^*$-moment map $\beta \alpha$ changes sign under the Weyl action.
The hypertoric variety $Y_{\SL(2,\C}$ is now just $\HH\cong \C^2$, which embeds into $Q_{\SL(2,\C)}$ by $(\alpha,\beta)\mapsto \left(\left(\begin{array}{c} \alpha\\ 0\end{array}\right), (\beta\ 0)\right)$ -- remark that this is not $W_{\SL(2,\C)}$-equivariant. 

\medskip
For the universal complex symplectic contraction, we need to consider the hyperk\"ahler quotient of two copies
of the implosion by the diagonal $U(1)$-action.
We can view this as $\C^4 \times (\C^4)^*$ with action\linebreak
$(v,w) \mapsto (t^{-1} v, t w)$. So we must impose the complex moment map
condition $\mu_{\C} := \langle v, w \rangle =0$ and factor out the complexified action.

\medskip
Again we can form the invariants $v \otimes w = (v_i w_j)_{i,j=1}^{4}$.
We obtain a hyperplane in the cone over the Segre variety defined by the vanishing of all $2 \times 2$ minors. If, for example, we set $v_1 = v_2 = w_3 = w_4 =0$
then the condition $\mu_\C =0$ is automatically satisfied and we
recover the (real) symplectic contraction given by coordinates
$v_3 w_1, v_3 w_2, v_4 w_1, v_4 w_2$ with one quadratic relation.

\medskip
More concretely, it is well-known that our
space can  be identified with the
Swann bundle ${\mathcal U} (Gr_{2} \C^4)$ of the Grassmannian of complex 2-planes in $\C^4$. Explicitly, $v \otimes w$ defines an element of the closure of the highest
root (or minimal) nilpotent orbit in $\sln(4, \C)$, and this is the total space of the Swann
bundle, see page 432 in \cite{swann.1991}.
If we reduce at a generic nonzero level we instead obtain the Calabi space
$T^* \C \PP^3$.

\medskip
In each case the hyperk\"ahler space  has an $\SU(4)$-action preserving the
hyperk\"ahler structure -- in the case of ${\mathcal U} (Gr_{2} \C^4)$ this is induced from the $\SU(4)$-action on the quaternionic K\"ahler Wolf space\linebreak
$Gr_{2} (\C^4) = \SU(4)/S(\Un(2) \times \Un(2))$.
This $\SU(4)$-action includes two commuting $\SU(2)$'s, as well as
a circle action commuting with each $\SU(2)$ factor, in accordance with
our interpretation of the space as a contraction for $\SU(2)$.

\medskip
In this case the Weyl group action on the complex symplectic contraction had already been observed by Swann, see Proposition 4.4.1 in \cite{Swann.dphil}, as the Galois action for the covering of a nilpotent orbit closure in $\sP(2,\C)$ by the highest root nilpotent orbit closure in $\sln(4,\C)$.  Indeed, if we identify $\C^4\otimes \C^4$ with $\gl(4)\cong \C^4\otimes (\C^4)^*$ by means of ${v\otimes w\mapsto v\otimes (\widetilde{J}w)^T}$ (where $\widetilde{J}=\left(\begin{array}{cc}J& 0 \\ 0 &  J\end{array}\right)$ and $J$ is as above), then $\sP(2)$ corresponds to $\operatorname{Sym}^2\C^4\subset \C^4\otimes \C^4$.  The action of $\gamma$ is then just the restriction of the transposition of the factors of $\C^4\otimes\C^4$, which on $\gl(4)$ translates to $A\mapsto \widetilde{J}A^T\widetilde{J}$.

\medskip
The locus $Q^R_{\circ}$ of \S \ref{complex} corresponds to the locus where the $T_{\C}=\C^*$-moment map $\beta \alpha = \alpha_1 \beta_1 + \alpha_2 \beta_2$ is nonzero. The associated locus in the contraction represented by points of $Q^R_{\circ} \times Q^L_{\circ}$ is therefore the 
set of points where the sum of the moment maps from the right and left implosions is zero, but the values of the individual maps are nonzero.

\begin{remark}
As described in Section 2, for $K=\SU(n)$ we have a quiver description of the implosion,
as the space of full flag quivers
\begin{equation*}
  0 \stackrel[\beta_0]{\alpha_0}{\rightleftarrows}
  \C \stackrel[\beta_1]{\alpha_1}{\rightleftarrows}
  \C^{2}\stackrel[\beta_2]{\alpha_2}{\rightleftarrows}\dots
  \stackrel[\beta_{n-2}]{\alpha_{n-2}}{\rightleftarrows} \C^{n-1}
  \stackrel[\beta_{n-1}]{\alpha_{n-1}}{\rightleftarrows} \C^{n},
\end{equation*}
with $\alpha_0=\beta_0=0$,
satisfying the equations
\begin{equation}  \label{eq:mmcomplex}
  \alpha_{i-1}\beta_{i-1} - \beta_i \alpha_i = \lambda^\C_i I \qquad
  (i=1,\dots,n-1),
\end{equation}
for free complex scalars \( \lambda^\C_1,\dots,\lambda^\C_{n-1} \),
modulo the natural action of 
$\prod_{i=1}^{n-1} \SL(i,\C)$
\begin{equation} \label{quiveract}
\alpha_i \mapsto g_{i+1} \alpha_i g_{i}^{-1}, \;\;\; 
\beta_i \mapsto g_i \beta_i g_{i+1}^{-1} \;\;\;(i=1, \ldots, n-2)
\end{equation}
\[
\alpha_{n-1} \mapsto \alpha_{n-1}g_{n-1}^{-1}, \;\;\;
\beta_{n-1} \mapsto g_{n-1} \beta_{n-1}.
\]
This can be identified with the hyperk\"ahler quotient of the space
of full flag quivers by the action of $\prod_{i=1}^{n-1} \SU(i)$.
Note that if $n=2$ this just reduces to $\Hom(\C,\C^2) \oplus \Hom(\C^2, \C)$ as in the example above since the
remaining equation \ref{eq:mmcomplex}) is automatic. For higher $n$ we do not in general have a simple description
of the space, although for $n=3$ it is the Swann bundle of the Grassmannian $SO(8)/(SO(4) \times SO(4))$, that is the closure of the 
minimal nilpotent orbit in $\so(8, \C)$.

We have a residual action of $\SL(n)$ by extending (\ref{quiveract}) to $i=n$, as well as a complex torus action
complexifying the hyperk\"ahler action of
of $T= \prod_{i=1}^{n-1} \Un(i)/\SU(i)$, whose moment map gives the $\lambda_i$.
In general, elements of the contraction may therefore be represented by $T$-equivalence classes of pairs of
quivers satisfying the equations, but with opposite $\lambda_i$.
\end{remark}

 \begin{remark}
In \cite{DKR} an alternative approach to implosion was introduced using
Nahm's equations
\[
\frac{dT_i}{dt} + [T_0,T_i]=[T_j,T_k] \;\;\; : \;\;\; (ijk) \;{\rm cyclic \;permutation \;of} \;(123).
\]
The idea here is to consider Nahm data on the half line
$[0,\infty)$ asymptotic to a commuting triple $(\tau_1,\tau_2,\tau_3)$
of elements of a fixed Cartan algebra. We further collapse by gauge transformations asymptotic at infinity to the commutator of the 
common centraliser $C$ of the triple (so no collapsing occurs on the
open dense set where $C$ is the maximal torus). We obtain a stratified pseudo-hyperk\"ahler space with an action of $T$, represented by gauge transformations asymptotic to values in $T$ at infinity. Moreover, the moment map for the $T$-action is just evaluation of the Nahm data at infinity, i.e. the triple $(\tau_1, \tau_2, \tau_3)$. The metric
becomes a genuine positive definite hyperk\"ahler metric on the level sets of the moment map, and the quotient gives the Kostant variety as required.

This construction will in general give a rather different space from the
implosion that we have discussed -- in particular the structure as an algebraic variety is still unclear. However, we may mimic the contraction construction and take the hyperk\"ahler quotient by $T$ of two copies of this space, and obtain a space with the correct dimension and action.
Being in the zero-level set of the $T$-action means the two sets of Nahm data are equal and opposite at infinity. 

Using the symmetry
$T_i(t) \mapsto-T_i(-t)$ of the Nahm equations, we may regard such configurations as giving Nahm data on two half lines that match at the respective points at infinity. Note that in general, we have a scaling symmetry $T_i(t) \mapsto cT_{i}(ct)$ of the Nahm equations that expands the
interval of definition by a factor of $c^{-1}$ if $0 < c < 1$, and under this symmetry the Nahm matrices scale by $c$ but their derivatives scale by $c^2$.
Recall that $T^*K_{\C}$ may be identified with the space of $\kf$-valued solutions to Nahm's equations smooth on a finite interval. 
Intuitively, we may view the above construction
as a limiting case of stretching out the interval so that the Nahm 
matrices are almost commuting in the interior.

 \end{remark}

\section*{Acknowledgements} We thank the organisers of the conference
"Moduli spaces and geometric structures" held in ICMAT Madrid in September 2022
for providing such a stimulating atmosphere.  We thank Victor Ginzburg and David Kazhdan for sending us a copy of \cite{ginzburg-kazhdan.unpubl}. We also thank the referee for several valuable suggestions.

\end{document}